\documentclass{amsart}
\usepackage{amsfonts}
\usepackage{mathrsfs}
\usepackage{graphicx}
\usepackage[abs]{overpic}
\usepackage{geometry}
\newtheorem{theorem}{Theorem}[section]

\newtheorem{proposition}[theorem]{Proposition}

\theoremstyle{definition}

\newtheorem{example}[theorem]{Example}

\newcommand{\R}{\mathbb{R}}
\newcommand{\Q}{\mathbb{Q}}
\newcommand{\Z}{\mathbb{Z}}
\newcommand{\tb}{\mathrm{tb}}
\newcommand{\tw}{\mathrm{tw}}
\newcommand{\rot}{\mathrm{rot}}
\newcommand{\st}{\mathrm{st}}
\newcommand{\Int}{\mathrm{Int}}

\theoremstyle{remark}

\numberwithin{equation}{section}

\begin{document}

\title{Legendrian torus knots in $S^1\times S^2$}

\author{Feifei Chen}
\address{School of Mathematical Sciences\\ Peking University\\ Beijing
100871, China}
\email{chernfeifei@163.com}

\author{Fan Ding}
\address{LMAM, School of Mathematical Sciences\\ Peking University\\ Beijing
100871, China}
\email{dingfan@math.pku.edu.cn}

\author{Youlin Li}
\address{Department of Mathematics \\Shanghai Jiao Tong University\\
Shanghai 200240, China}
\email{liyoulin@sjtu.edu.cn}





\begin{abstract}
We classify the Legendrian torus knots in $S^1\times S^2$ with its standard tight contact structure up to Legendrian isotopy.
\end{abstract}

\maketitle

\section{Introduction}

Considerable progress has been made towards the classification of
Legendrian knots and links in contact 3-manifolds. The
classification of Legendrian unknots (in any tight contact
3-manifolds) is due to Eliashberg and Fraser, cf. \cite{EF1} and
\cite{EF2}. The classification of Legendrian torus knots and the
figure eight knot in $S^3$ with its standard tight contact
structure is due to Etnyre and Honda \cite{EH}. For a general
introduction to this topic see \cite{Et}. Recently, Legendrian
twist knots are classified in \cite{ENV} and Legendrian cables of
positive torus knots are classified in \cite{ELT}. Legendrian
knots in contact 3-manifolds other than $S^3$ (with its standard
contact structure) are also studied. For example, Legendrian
linear curves in all tight contact structures on $T^3$ are
classified in \cite{Gh} and Legendrian torus knots in the $1$-jet
space $J^1(S^1)$ with its standard tight contact structure are
classified in \cite{DG1}. Legendrian rational unknots in lens
spaces are studied in \cite{BE} and \cite{GO}. Legendrian torus
knots in lens spaces are studied in \cite{On}. The purpose of the
present paper is to give a complete classification of Legendrian
torus knots in $S^1\times S^2$ with its standard tight contact
structure.

The standard tight contact structure $\xi_{\st}$ on $S^{1}\times S^{2}\subset S^{1}\times \R^3$
is given by $\ker (x_3\mathrm{d}\theta+x_1\mathrm{d}x_2-x_2\mathrm{d}x_1)$,
where $\theta$ denotes the $S^1$-coordinate and $(x_1,x_2,x_3)$ cartesian coordinates on $\R^3$.
Here we think of $S^1$ as $\R /2\pi \Z$. This is the unique positive tight contact structure on $S^1\times S^2$
up to isotopy; see \cite[Theorem 4.10.1]{Ge}. Moreover, $\xi_{\st}$ is trivial as an abstract real $2$-plane bundle.
Suppose $K$ is an oriented Legendrian knot in $(S^1\times S^2,\xi_{\st})$.
For any preassigned choice of nowhere zero vector field $v$ in $\xi_{\st}$
(up to homotopies through such vector fields) we can define the rotation number
$\rot_v(K)$ to be the signed number of times that the tangent vector field $\tau$ to $K$ rotates in $\xi_{\st}$
relative to $v$ as we travel once around $K$ in the direction specified by its orientation. Usually we omit $v$ in $\rot_v(K)$.

Let $T_0=\{(\theta,x_1,x_2,x_3)\in S^1\times S^2: x_3= 0\}$, then
$T_0$ is a Heegaard torus, i.e., the closures of the components of $(S^1\times S^2)\setminus T_0$
are two solid tori. A knot in $S^1\times S^2$ is called a torus knot if it is (smoothly) isotopic to a knot on $T_0$.
Consider the solid torus $V_0=\{(\theta,x_1,x_2,x_3)\in S^1\times S^2: x_3\ge 0\}$.
The curve on $T_0=\partial V_0$ given by $\theta=0$, oriented positively in the $x_1x_2$-plane,
is a meridian of $V_0$, and let $m_0\in H_1(T_0)$ denote the class of this meridian;
the curve given by $(x_1,x_2,x_3)=(1,0,0)$, oriented by the parameter $\theta$,
is a longitude, and let $l_0\in H_1(T_0)$ denote the class of this longitude.
Then $(m_0,l_0)$ is a basis for $H_1(T_0)$. An oriented knot in $S^1\times S^2$
is called a $(p,q)$-torus knot if it is isotopic to an oriented knot on $T_0$
homologically equivalent to $pm_0+ql_0$ (with $p$ and $q$ coprime). A $(\pm 1,0)$-torus knot is trivial,
i.e., bounds a disk in $S^1\times S^2$. A $(p,1)$-torus knot is isotopic to $S^1\times \{(0,0,1)\}$,
oriented by the parameter $\theta$. For this knot type, we have

\begin{theorem}\label{S^1 knot} Two oriented Legendrian knots
in $(S^1\times S^2,\xi_{\st})$ of the same oriented knot type as $S^1\times \{ (0,0,1)\}$
are Legendrian isotopic if and only if their rotation numbers agree.
\end{theorem}

Let $K$ be a $(p,q)$-torus knot with $q\ge 2$. There is a Heegaard torus $T$ on which $K$ sits.
In Section~2, we shall prove that the framing of $K$ induced by $T$
(i.e., the framing corresponding to two simple closed curves which are the intersection of $T$
and the boundary of a tubular neighborhood of $K$), is independent of the Heegaard torus $T$ we choose.
For a Legendrian $(p,q)$-torus knot $K$ with $q\ge 2$, we define the twisting number $\tw (K)$ to
be the number of counterclockwise (right) $2\pi$ twists of $\xi_{\st}$ along $K$, relative to the framing of $K$ induced by $T$. We have

\begin{theorem}\label{torus knot} Let $K$ be a Legendrian $(p,q)$-torus knot and $K'$ a Legendrian $(p',q)$-torus knot
in $(S^1\times S^2,\xi_{\st})$ with $q\ge 2$. Then $K$ and $K'$ are Legendrian isotopic if and only if
their oriented knot types and their classical invariants $\tw$ and $\rot$ agree.
\end{theorem}

In Section~2, we give the topological classification of torus knots in $S^1\times S^2$. In Section~3,
we recall the classification of Legendrian torus knots in the 1-jet space $J^1(S^1)$ of the circle with its
standard contact structure, and give an analogous result for Legendrian torus knots in a solid torus with
a suitable tight contact structure. In Section~4, we prove Theorems~\ref{S^1 knot} and~\ref{torus knot}.\\

{\bf Acknowledgements}:  Authors would like to thank John Etnyre for helpful conversations, careful reading
of the paper and invaluable comments.  Part of this work was carried out while the third author was
visiting Georgia Institute of Technology and he would like to thank them for their hospitality.
The third author was partially supported by NSFC grant 11001171 and the China Scholarship Council grant 201208310626.

\section{Topological torus knots in $S^1\times S^2$}

For $r_1,r_2\in \Q\setminus \Z$, we describe the Seifert manifold
$M(D^2;r_1,r_2)$  as follows. Let $\Sigma$ be an oriented pair of
pants. For each connected component $T_i$ of
$-(\partial\Sigma\times S^1)=T_1\cup T_2\cup T_3$, denote the
homology class in $H_1(T_i)$ of the connected component of
$-\partial(\Sigma\times\{ 1\})$ in $T_i$ by $\mu_i$ and the
homology class of the $S^1$ factor in $H_1(T_i)$ by $\lambda_i$.
For $i=1,2$, let $V_{i}=D^{2}\times S^{1}$. Then $M(D^2;r_1,r_2)$
is obtained from $\Sigma\times S^1$ by gluing $V_i$ to $T_i$,
$i=1,2$, using a diffeomorphism $\varphi_i:\partial V_i\to T_i$
sending the meridian $\partial (D^2\times\{ 1\})$ to a circle in
$T_i$ homologically equivalent to $p_i\mu_i-q_i\lambda_i$, where
$p_i,q_i$ are coprime and $\frac{q_i}{p_i}=r_i$. Note that
$M(D^2;r_1,r_2)$  corresponds to $M(0,1;r_1,r_2)$ in \cite[Chapter
2, Section 1]{Ha}.

Let $K$ be an oriented knot in $S^1\times S^2$. Denote a tubular
neighborhood of $K$ (diffeomorphic to a solid torus) by $\nu K$.
Let $T$ be a Heegaard torus in $S^1\times S^2$ on which $K$ sits.
Then $\partial(\nu K)\cap T$ are the two longitudes of $\nu K$
determined by the framing of $K$ induced by $T$.

\begin{proposition} \label{framing}
If $K$ is a $(p,q)$-torus knot in $S^1\times S^2 $ with $q\ge 2$,
then the framing of $K$ induced by a Heegaard torus $T$ on which
$K$ sits is independent of the Heegaard torus $T$ we choose.
\end{proposition}

\begin{proof}
The compact manifold $S^1\times S^2 \setminus \Int(\nu K)$ is a
Seifert fibred space having $\partial(\nu K)\cap T$ as two regular
fibers. Fix such a Seifert fibration of $S^1\times S^2 \setminus
\Int(\nu K)$, then the essential surfaces are either vertical or
horizontal (cf. \cite[Proposition~1.11]{Ha}). If $q>2$, then
$S^1\times S^2 \setminus \Int(\nu K)$ has a unique (up to isotopy)
essential surface which is a vertical annulus and separating. If
$q=2$, then $S^1\times S^2 \setminus \Int(\nu K)$ has a unique
vertical essential annulus and a horizontal essential annulus. The
vertical annulus is separating, and the horizontal one is
non-separating. So the Seifert fibred space $S^1\times S^2
\setminus \Int(\nu K)$ has a unique essential separating annulus
whose boundary is isotopic to $\partial(\nu K)\cap T$ in
$\partial(\nu K)$. Hence the framing of $K$ induced by a Heegaard
torus $T$ on which $K$ sits is determined by the the compact
manifold $S^1\times S^2 \setminus \Int(\nu K)$, and thus
independent of the Heegaard torus $T$ we choose.
\end{proof}

We classify the torus knots in $S^1\times S^2$ as follows.

\begin{proposition} \label{topological classification}
For $q\ge 2$, a $(p,q)$-torus knot and a $(p',q)$-torus knot are isotopic if and only if $p'\equiv p\mod 2q$ or $p'\equiv -p\mod 2q$.
\end{proposition}

\begin{proof} For $p,q$ coprime and $q\ge 2$, let $K(p,q)$ be an oriented knot on
$T_0\subset S^1\times S^2$ homologically equivalent to $pm_0+ql_0$
(for the definitions of $T_0,m_0,l_0$, see Section~1). We prove
that $K(p,q)$ and $K(p',q)$ are isotopic in $S^1\times S^2$ if and
only if $p'\equiv p\mod 2q$ or $p'\equiv -p\mod 2q$. We divide the
proof into 4 steps.

\textbf{Step 1}. We prove that if $p'\equiv p\mod 2q$ or $p'\equiv -p\mod
2q$, then $K(p',q)$ and $K(p,q)$ are isotopic.

Write $r_{\theta}$ for the rotation of $S^2\subset\R^3$ about the
$x_3$-axis through an angle $\theta$. Define a diffeomorphism $r$
of $S^1\times S^2$ by
$r(\theta,\mathbf{x})=(\theta,r_{\theta}(\mathbf{x}))$. Since
$r^2(\theta,\mathbf{x})=(\theta,r_{2\theta}(\mathbf{x}))$ and
$\pi_1(SO(3))\cong \Z_2$, $r^2$ is isotopic to the identity (cf.
\cite[Lemma 1]{DG3}). The diffeomorphism $r^2$ sends a knot on $T_0$ homologically
equivalent to $pm_0+ql_0$ to a knot on $T_0$ homologically
equivalent to $(p+2q)m_0+ql_0$. Thus $K(p,q)$ is isotopic to
$K(p+2q,q)$ in $S^1\times S^2$. Hence if $p'\equiv p\mod 2q$, then
$K(p',q)$ is isotopic to $K(p,q)$ in $S^1\times S^2$.

Define a diffeomorphism $b$ of $S^1\times S^2$ by
$b(\theta,x_1,x_2,x_3)=(\theta,x_1,-x_2,-x_3)$. Then $b$ is isotopic to
the identity and sends a knot on $T_0$ homologically equivalent to
$pm_0+ql_0$ to a knot on $T_0$ homologically equivalent to
$-pm_0+ql_0$. Thus $K(p,q)$ is isotopic to $K(-p,q)$ in $S^1\times
S^2$. Combining this with the preceding paragraph, we conclude
that if $p'\equiv -p\mod 2q$, then $K(p',q)$ is isotopic to
$K(p,q)$ in $S^1\times S^2$.

\textbf{Step 2}. We prove that if $K(p,q)$ and $K(p',q)$ are isotopic in
$S^1\times S^2$, then $p'\equiv p\mod q$ or $p'\equiv -p\mod q$.

Since $p$ and $q$ are coprime, we may choose $s,t\in\Z$ such that
$ps-tq=1$. The closure of the complement of a tubular neighborhood
of $K(p,q)$ in $S^1\times S^2$ is the Seifert manifold
$M(D^2;\frac{s}{q},-\frac{s}{q})$. Similarly, we choose
$s',t'\in\Z$ such that $p's'-t'q=1$. Then the closure of the
complement of a tubular neighborhood of $K(p',q)$ is
$M(D^2;\frac{s'}{q},-\frac{s'}{q})$. Now suppose $K(p,q)$ and
$K(p',q)$ are isotopic in $S^1\times S^2$. Then
$M(D^2;\frac{s}{q},-\frac{s}{q})$ and
$M(D^2;\frac{s'}{q},-\frac{s'}{q})$ are orientation-preserving
diffeomorphic. By \cite[Theorem 2.3 and Proposition 2.1]{Ha}, we
have $s'\equiv s\mod q$ or $s'\equiv -s\mod q$. If $s'\equiv s\mod
q$, then there exists an integer $k$ such that $s'=s+kq$. Combined
with $ps-tq=1$ and $p's'-t'q=1$, this gives $(p'-p)s=(t'-t-kp')q$.
Since $q$ and $s$ are coprime, $q$ divides $p'-p$. Thus $p'\equiv
p\mod q$. Similarly, if $s'\equiv -s\mod q$, then $p'\equiv -p\mod
q$.

\textbf{Step 3}. We prove that for $q>2$, $K(p,q)$ is not isotopic to
$K(p+q,q)$ in $S^1\times S^2$.

First note that $r(K(p,q))$ is isotopic to $K(p+q,q)$. Thus if
$K(p,q)$ is isotopic to $K(p+q,q)$, then $r(K(p,q))$ is isotopic
to $K(p,q)$ and there is an orientation-preserving diffeomorphism
$g$, isotopic to $r$, such that the restriction of $g$ on $K(p,q)$ is the
identity. Note that $g(T_0)$ is also a Heegaard torus. Thus the framing of
$K(p,q)$ induced by $T_0$ is the same as the framing of $K(p,q)$
induced by $g(T_0)$. Hence after an isotopy, we may assume that
$g$ is the identity on a tubular neighborhood $N$ (diffeomorphic
to a solid torus) of $K(p,q)$.

The compact manifold $(S^1\times S^2)\setminus\Int(N)$ is diffeomorphic to the
Seifert manifold $M(D^2;\frac{s}{q},-\frac{s}{q})$. Recall that
$M(D^2;\frac{s}{q},-\frac{s}{q})$ is obtained from $\Sigma\times
S^1$ by gluing $V_i$ to $T_i$ ($i=1,2$), using a diffeomorphism
$\varphi_i:\partial V_i\to T_i$ sending the meridian $\partial
(D^2\times\{ 1\})$ to a circle in $T_i$ homologically equivalent
to $q\mu_1-s\lambda_1$ for $i=1$, or $q\mu_2+s\lambda_2$ for $i=2$
(for the notation $\Sigma
,T_i(i=1,2,3),\mu_i,\lambda_i,V_i(i=1,2)$, see the first paragraph
of this section). The simple closed curve $S^1\times \{ (0,0,-1)\}$ corresponds to a core
of $V_1$ and the simple closed curve $S^1\times \{(0,0,1)\}$ corresponds to a core of
$V_2$. Consider the restriction of $g$ to $(S^1\times
S^2)\setminus\Int(N)$, still denoted by $g$, as a
self-diffeomorphism of $M(D^2;\frac{s}{q},-\frac{s}{q})$ which is
the identity on the boundary. The compact surface $A=T_0\setminus\Int (N)$ is an
essential vertical annulus in $M(D^2;\frac{s}{q},-\frac{s}{q})$.
By \cite[Proposition 1.11]{Ha}, $g(A)$ is isotopic (relative to
the boundary) to a vertical annulus. Thus we may assume that
$g(A)$ is a vertical annulus (disjoint from $V_1$ and $V_2$).

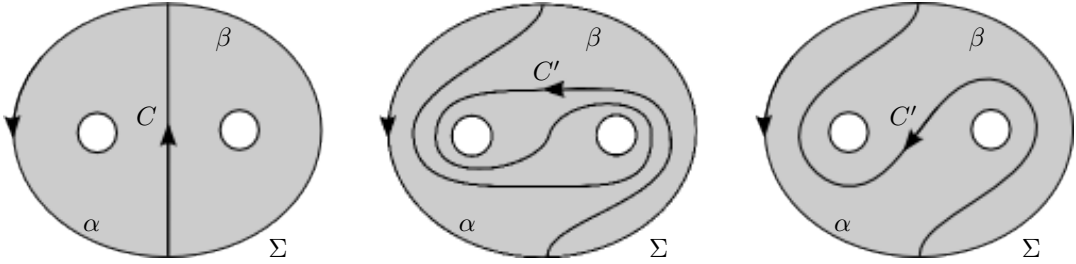
\begin{figure}[htb]
\begin{overpic}
{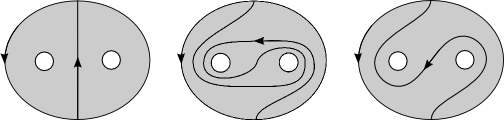}
\put(30, 10){$\alpha$}
\put(80, 80){$\beta$}
\put(50, 50){$C$}
\put(100, 0){$\Sigma$}

\put(172, 10){$\alpha$}
\put(220, 80){$\beta$}
\put(200, 67){$C'$}
\put(244, 0){$\Sigma$}

\put(314, 10){$\alpha$}
\put(365, 80){$\beta$}
\put(335, 50){$C'$}
\put(385, 0){$\Sigma$}
\end{overpic}
\caption{The oriented pair of pants $\Sigma$, the oriented arc $C$ shown in the left, and two possible oriented arcs $C'$ shown in the middle and right.}
\label{fig:1}
\end{figure}

Let $C$ denote the arc in $\Sigma$ such that $C\times S^1$ is the
annulus $A$ (see the left of Figure~\ref{fig:1}). Let $C'$ denote
the arc in $\Sigma$ such that $C'\times S^1$ is the annulus $g(A)$
(see the middle and  right of  Figure~\ref{fig:1}). Let $B$ denote
the component of $\partial\Sigma$ such that $B\times S^1$ is
$T_3$. The two points $C\cap B$ divides $B$ into $2$ arcs $\alpha$
and $\beta$ (see the left of Figure~\ref{fig:1}). In
$M(D^2;\frac{s}{q},-\frac{s}{q})$, the torus $(\alpha \cup
C)\times S^1$ bounds a solid torus $N_0$ containing one of $V_1$
and $V_2$, say $V_1$. Orient $\alpha$ as a part of $\partial
\Sigma$. Orient $C$ such that the orientation on $\alpha$ and the
orientation on $C$ give an orientation on $\alpha\cup C$ (see the
left of Figure~\ref{fig:1}). Note that $g$ is the identity on
$T_3$. Orient $C'$ such that the orientation on $\alpha$ and the
orientation on $C'$ give an orientation on $\alpha\cup C'$ (see
the middle and right of Figure~\ref{fig:1}). Denote the class in
$H_1((\alpha\cup C)\times S^1)$ of $(\alpha\cup C)\times\{ 1\}$ by
$\mu$ and the class in $H_1((\alpha\cup C)\times S^1)$ of a fiber
by $\lambda$. Then $q\mu-s\lambda$ is the class of a meridian of
$N_0$. Denote the class in $H_1((\alpha\cup C')\times S^1)$ of
$(\alpha\cup C')\times\{ 1\}$ by $\mu'$ and the class in
$H_1((\alpha\cup C')\times S^1)$ of a fiber by $\lambda'$. In
$M(D^2;\frac{s}{q},-\frac{s}{q})$, the torus $(\alpha \cup
C')\times S^1$ bounds a solid torus $N_0^{\prime}$ containing one
of $V_1$ and $V_2$, and $g$ sends $N_0$ onto $N_0^{\prime}$.

Suppose that $g(C\times\{ 1\})\subset C'\times S^1$ wraps around
the $S^1$ factor $k$ times as we travel once around $C$. If
$N_0^{\prime}$ contains $V_1$ (see the middle of
Figure~\ref{fig:1}), then $q\mu'-s\lambda'$ is the class of a
meridian of $N_0^{\prime}$. The diffeomorphism $g$ sends the class
$q\mu-s\lambda\in H_1((\alpha\cup C)\times S^1)$ to the class
$q\mu'+(kq-s)\lambda'\in H_1((\alpha\cup C')\times S^1)$. Since
$q\mu'+(kq-s)\lambda'$ needs to be the class of a meridian of
$N_0^{\prime}$, we have $k=0$. Thus we may assume that $g$ is the
identity on $V_1$ (cf. \cite[Section 4.4]{HR}). Then $g$, as a
self-diffeomorphism of $S^1\times S^2$, is the identity near
$S^1\times\{ (0,0,-1)\}$. Hence by \cite[Section 4.4]{HR}, $g$ is
isotopic to the identity. But $g$ is isotopic to $r$ and $r$ is
not isotopic to the identity (cf. \cite[Lemma 1]{DG3}), and we get
a contradiction. Hence $N_0^{\prime}$ contains $V_2$ (see the
right of Figure~\ref{fig:1}). Then $q\mu'+s\lambda'$ is the class
of a meridian of $N_0^{\prime}$ and $kq-s=s$. Thus $2s$ is divided
by $q$. Since $q,s$ are coprime, we conclude that $2$ is divided
by $q$, contrary to the assumption that $q>2$.

\textbf{Step 4}. If $K(p,q)$ and $K(p',q)$ are isotopic in $S^1\times S^2$,
then $p'\equiv p\mod 2q$ or $p'\equiv -p\mod 2q$.

First assume that $q>2$. If $K(p,q)$ and $K(p',q)$ are isotopic,
then by Step 2, there exists an integer $k$ such that $p'=p+kq$ or
$p'=-p+kq$. If $k$ is odd, then by Step 1, $K(p',q)$ is isotopic
to $K(p+q,q)$. Hence $K(p+q,q)$ is isotopic to $K(p,q)$ in
$S^1\times S^2$, contrary to the conclusion in Step 3. Thus $k$ is
even and $p'\equiv p\mod 2q$ or $p'\equiv -p\mod 2q$.

Assume now that $q=2$. If $K(p,2)$ and $K(p',2)$ are isotopic in
$S^1\times S^2$, then by Step 2, there exists an integer $k$ such
that $p'=p+2k$. If $k$ is even, then $p'\equiv p\mod 4$. If $k$ is
odd, then $p+k$ is even since $p$ is odd ($p$ and $2$ are
coprime). Hence by $p'=-p+2(p+k)$, we have $p'\equiv -p\mod 4$.
\end{proof}

\section{Legendrian torus knots in $J^1(S^1)$ and in a solid torus}

For fixing notation, we give definitions and properties of
Legendrian $(p,q)$-torus knots in $J^1(S^1)$ and in a solid torus.

\subsection{Legendrian torus knots in $J^1(S^1)$}
Let $J^1(S^1)=T^*S^1\times\R =S^1\times \R^2=\{ (\theta,y,z):\theta\in S^1=\R/2\pi\Z,y,z\in\R\}$
be the $1$-jet space of $S^1$ with its standard contact structure $\xi_0=\ker(\mathrm{d}z-y\mathrm{d}\theta)$.

One can visualize a Legendrian knot $K\subset J^1(S^1)$ in its front projection to
a strip $[0,2\pi]\times \R$ in the $\theta z$-plane. The Thurston-Bennequin invariant
of $K$ is $\tb(K)=\mathrm{writhe}(K)-\frac{1}{2}\# (\mathrm{cusps}(K))$, where the quantities
on the right are computed from the front projection of $K$. This invariant has a definition
that does not rely on the front projection, and which shows that $\tb$ is a Legendrian isotopy invariant, cf. \cite{DG1}.

For an oriented Legendrian knot $K$ in $J^1(S^1)$, we may define its rotation number
in terms of its front projection as $\rot(K)=\frac{1}{2}(c_--c_+)$, with $c_{\pm}$ the number
of cusps oriented upwards or downwards, respectively; cf. \cite{DG2}. This is the same as
the rotation number defined by the nowhere zero vector field $\partial_y$ in $\xi_0$, cf. \cite[Section 6]{DG1}.

By a torus knot in $J^1(S^1)$, we mean a knot that sits on a torus isotopic to
the torus $T_1=\{(\theta,y,z)\in J^1(S^1):y^2+z^2=1\}$. Consider the solid torus $M_1=\{(\theta,y,z)\in J^1(S^1):y^2+z^2\leq1\}$.
The curve on $T_1=\partial M_1$ given by $\theta=0$, oriented positively in the $yz$-plane,
is a meridian of $M_1$, and let $m_1\in H_1(T_1)$ denote the class of this meridian; the curve
given by $(y,z)=(1,0)$, oriented by the parameter $\theta$, is a longitude, and
let $l_1\in H_1(T_1)$ denote the class of this longitude. Then $(m_1,l_1)$ is a positive basis for $H_1(\partial M_1)$.

For $p,q$ coprime, a $(p,q)$-torus knot in $J^1(S^1)$ is an oriented knot
isotopic to an oriented knot on $T_1$ homologically equivalent to $pm_1+ql_1$.
A $(\pm 1,0)$-torus knot in $J^1(S^1)$ is trivial. A $(p,1)$-torus knot in $J^1(S^1)$
is isotopic to $S^1\times \{(0,0)\}$, oriented by the variable $\theta$. For $q\ge 2$,
if a $(p,q)$-torus knot is isotopic to a $(p',q)$-torus knot in $J^1(S^1)$, then $p=p'$.
This can be seen as follows. Let $K$ and $K'$ be a $(p,q)$-torus knot and a $(p',q)$-torus knot
in $J^1(S^1)$, respectively. Embed $J^1(S^1)$ into $S^3$ as an open tubular
neighborhood of an unknot in $S^3$. Then $K$ and $K'$ become a $(p+cq,q)$-torus knot and a $(p'+cq,q)$-torus
knot in $S^3$ for some $c\in\Z$. Using different framings of the unknot to define the embedding,
$c$ can be any integer. Thus if $K$ and $K'$ are isotopic in $J^1(S^1)$, then the corresponding
$(p+cq,q)$-torus knot and $(p'+cq,q)$-torus knot are isotopic in $S^3$ for each $c\in \Z$.
Thus by the classification of torus knots in $S^3$, we have $p=p'$.

For an oriented Legendrian knot $K$ in a contact $3$-manifold,
we have a positive stabilization $S_+(K)$ and a negative stabilization $S_-(K)$ (cf. \cite{EH}).
Stabilizations are well defined and commute with each other. For an oriented Legendrian knot $K$ in $J^1(S^1)$,
we have $\tb(S_{\pm}(K))=\tb(K)-1$, $\rot(S_{\pm}(K))=\rot (K)\pm 1$.
Notice that the stabilization affects the classical invariants in this way for any oriented Legendrian knot $K$
in any contact $3$-manifold, as long as the classical invariants can be defined (cf. \cite{EH}).
The results in the following three paragraphs can be deduced from \cite[Proof of Theorem 3.3]{DG1}.

The maximum Thurston-Bennequin invariant of a Legendrian knot in $J^1(S^1)$ isotopic to $S^1\times\{(0,0)\}$ is $0$.
Any Legendrian knot isotopic to $S^1\times\{(0,0)\}$ with $\tb=0$ is Legendrian isotopic to
$S^1\times\{(0,0)\}$. A Legendrian knot in $J^1(S^1)$ isotopic to $S^1\times\{(0,0)\}$ with non-maximum $\tb$ can be destabilized in $J^1(S^1)$.

For $p\ge 1$ and $q\ge 2$,  the maximum Thurston-Bennequin invariant of a Legendrian $(p,q)$-torus
knot in $J^1(S^1)$ is $p(q-1)$. Any two Legendrian $(p,q)$-torus knots with maximum Thurston-Bennequin
invariant are Legendrian isotopic. A Legendrian $(p,q)$-torus knot in $J^1(S^1)$ with non-maximal $\tb$ can be destabilized in $J^1(S^1)$.

For $p<0$ and $q\ge 2$, the maximum Thurston-Bennequin invariant of a Legendrian $(p,q)$-torus knot
in $J^1(S^1)$ is $pq$. The possible values of $\rot$ (for $\tb=pq$ being maximum) are shown to lie
in $\{ \pm(p+2lq):l\in\Z,0\le l<-\frac{p}{q}\}$. A Legendrian $(p,q)$-torus knot in $J^1(S^1)$
with non-maximal $\tb$ can be destabilized in $J^1(S^1)$.

By \cite[Theorem 3.3]{DG1}, two oriented Legendrian torus knots in $J^1(S^1)$
are Legendrian isotopic if and only if their oriented knot types and their classical invariants $\tb$ and $\rot$ agree.

\subsection{Legendrian torus knots in a solid torus}

Let $V$ be an oriented solid torus. Let $m\in H_1(\partial V)$ be the class of an oriented meridian
of $V$ and $l\in H_1(\partial V)$ the class of an oriented longitude of $V$. The meridian and the
longitude are oriented in such a way that $m,l$ form a positive basis for $H_1(\partial V)$.
By a torus knot in $V$, we mean a knot in $\Int(V)$ that sits on a torus parallel to $\partial V$
(i.e. this torus and $\partial V$ bound a thickened torus in $V$). For $p,q$ coprime, a $(p,q)$-torus knot in $V$
is an oriented knot in $\Int (V)$ that sits on a torus $T$ parallel to $\partial V$
such that this oriented knot is homologically equivalent to $pm+ql$ in the thickened torus bounded by $T$ and $\partial V$.

Similar to the cases in $J^1(S^1)$, we have: a $(\pm 1,0)$-torus
knot in $V$ is trivial; a $(p,1)$-torus knot in $V$ is isotopic to
a core of $V$; for $q\ge 2$, if a $(p,q)$-torus knot in $V$ is
isotopic to a $(p',q)$-torus knot in $V$, then $p=p'$. For a
$(p,q)$-torus knot $K$ with $q\ge 2$, the framing of $K$ induced
by a torus $T$ parallel to $\partial V$ on which $K$ sits is
independent of the torus $T$ we choose. This can be seen by
embedding $V$ in $S^1\times S^2$ and using
Proposition~\ref{framing}.

Now let $\xi$ be a positive tight contact structure on $V$ with
convex boundary having two dividing curves each in the homology
class $l$. For a Legendrian $(p,q)$-torus knot $K$ in $(V,\xi)$
with $q\ge 2$, we define the twisting number $\tw(K)$ to be the
number of counterclockwise (right) $2\pi$ twists of $\xi$ along
$K$, relative to the framing of $K$ induced by a torus $T$
parallel to $\partial V$ on which $K$ sits. Since $\xi$ is trivial
as an abstract real $2$-plane bundle, using a nowhere zero vector
field in $\xi$, we can define the rotation number of an oriented
Legendrian knot in $(V,\xi)$.

Let $L$ be an oriented Legendrian core of $(V,\xi)$ such that $l$
is the class of a parallel copy of $L$ determined by the contact
framing. Such a Legendrian core exists by \cite[Theorem 3.14]{EH}.
Furthermore, by \cite[Theorem 3.14]{EH}, we have a contact
embedding $\phi$ from $(V,\xi)$ to $(J_1(S^1),\xi_0)$ whose image
is $M_1$ (possibly perturbing $\partial V$) and sending $l$ to
$l_1$.

For a Legendrian $(p,q)$-torus knot $K$ in $(V,\xi)$ with $q\ge
2$, $\phi(K)$ is a Legendrian $(p,q)$-torus knot in
$(J_1(S^1),\xi_0)$, and $\tw(K)=\tb(\phi(K))-pq$. Using the
nowhere vanishing vector field in $(V,\xi)$, which is sent to
$\partial_{y}$ by $\mathrm{d}\phi$, to define the rotation number
of $K$, we have $\rot(K)=\rot(\phi(K))$.

The following proposition is essentially contained in \cite{EH}, and can be easily derived from Subsection 3.1 by a contact flow.

\begin{proposition} \label{classification in a solid torus}
With $V,\xi,L$ as above,
\begin{enumerate}
\item for $p,q$ coprime and $q\ge 2$, two Legendrian $(p,q)$-torus knots in $(V,\xi)$ are Legendrian isotopic
if and only if their classical invariants (twisting number and rotation number) agree;
\item an oriented Legendrian knot in $(\Int(V),\xi)$ isotopic to $L$ is Legendrian isotopic to a stabilization of $L$.
\end{enumerate}
\end{proposition}

\section{Legendrian torus knots in $S^1\times S^2$}

First, we prove the main results.

\begin{proof}[Proof of Theorem~\ref{S^1 knot}] Let $K$ and $K'$ be two oriented
Legendrian knots in $(S^1\times S^2,\xi_{\st})$ which have the
same oriented knot type as $S^1\times \{ (0,0,1)\}$ and the same
rotation number. The knot $K$ has a tubular neighborhood $N$ with
convex boundary having two dividing curves each in the homology
class $\lambda$, where $\lambda$ is the class in $H_1(\partial N)$
of a parallel copy of $K$ determined by the contact framing.
Similarly, the knot $K'$ has a tubular neighborhood $N'$ with
convex boundary having two dividing curves each in the homology
class $\lambda'$, where $\lambda'$ is the class in $H_1(\partial
N')$ of a parallel copy of $K'$ determined by the contact framing.
This allows one to find a contactomorphism $\phi:N\to N'$ sending
$K$ to $K'$ and $\lambda$ to $\lambda'$ (and sending a meridian of
$N$ to a meridian of $N'$). Since a meridian of $N$ (respectively,
$N'$) is also a meridian of $(S^1\times S^2)\backslash \Int(N)$
(respectively, $(S^1\times S^2)\backslash \Int(N')$), $\phi$ can
be extended to a diffeomorphism of $S^1\times S^2$. Furthermore,
by Theorem 3.14 of \cite{EH}, $\phi$ can be extended to a
contactomorphism, still denoted by $\phi$, of $(S^1\times
S^2,\xi_{\st})$.

In \cite{DG3}, we have a contactomorphism $r_c$ of $(S^1\times S^2,\xi_{\st})$
isotopic to the diffeomorphism $r$ (for the definition of $r$, see the proof of Proposition \ref{topological classification})
such that for the oriented Legendrian knot $K_0$, which is shown in Figure~\ref{fig:2},
in $(S^1\times S^2,\xi_{\st})$, $r_c(K_0)$ is Legendrian isotopic to the positive stabilization $S_+(K_0)$ of $K_0$.
In particular, $\rot (r_c(K_0))=\rot (K_0)+1$. Thus for an oriented Legendrian
knot $K_1$ in $(S^1\times S^2,\xi_{\st})$ homotopic to $q$ times the standard generator
of the fundamental group $\pi_1(S^1\times S^2)\cong\Z$, we have $\rot(r_c(K_1))=\rot(K_1)+q$.
According to \cite{DG3}, any contactomorphism of $(S^1\times S^2,\xi_{\st})$ acting trivially on
homology is contact isotopic to a uniquely determined integer power of $r_c$. So $\phi$ is contact
isotopic to $r_c^m$ for some integer $m$. Hence $\rot (K')=\rot (K)+m$. Since $\rot (K)=\rot(K')$,
we conclude that $m=0$ and thus $\phi$ is contact isotopic to the identity. Thus $K$ and $K'$ are Legendrian isotopic.
\end{proof}

\begin{figure}[htb]
\begin{overpic}
{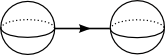}

\end{overpic}
\caption{The Legendrian knot $K_0$ in $(S^1\times S^2,\xi_{\st})$.}
\label{fig:2}
\end{figure}

\begin{proof}[Proof of Theorem~\ref{torus knot}] Perturb $T_0$ to be a convex torus $T_0'$
with two dividing curves each in the homology class $l_0'$, where
$l_0'\in H_1(T_0')$ corresponds to $l_0\in H_1(T_0)$ under the
perturbation. Let $m_0'\in H_1(T_0')$ denote the class
corresponding to $m_0\in H_1(T_0)$ under the perturbation. For the
notation $T_0,l_0,m_0$, see Section 1. Let $N_0$ and $N_1$ denote
the closures of the components of $(S^1\times S^2)\backslash T_0'$
containing $S^1\times \{ (0,0,1)\}$ and $S^1\times\{ (0,0,-1)\}$,
respectively. Let $K_0$ (respectively, $K_1$) be an oriented
Legendrian core of $N_0$ (respectively, $N_1$) such that $l_0'$ is
the class of a parallel copy of $K_0$ (respectively, $K_1$)
determined by the contact framing. One may consider $K_0$ as the
Legendrian knot $K_0$ shown in Figure~\ref{fig:2} and $K_1$ as a
Legendrian push-off of $K_0$ along the vertical direction. In
particular, $\rot (K_0)=\rot (K_1)$.

Let $K_2$ be a Legendrian knot in $(S^1\times S^2,\xi_{\st})$ which
sits on a Heegaard torus $T$. Denote the closures of the components of
$(S^1\times S^2)\backslash T$ by $N,N'$. Let $K_2'$ be an oriented Legendrian
core of $N'$ such that $K_2'$ is isotopic to $S^1\times \{ (0,0,1)\}$ in $S^1\times S^2$
as oriented knots. Stabilize $K_2'$ if necessary to make $\rot(K_2')=\rot(K_1)$.
By Theorem \ref{S^1 knot}, $K_2'$ is Legendrian isotopic to $K_1$. Thus we may assume
that $K_2$ is in $(S^1\times S^2)\backslash K_1$. Since $(S^1\times S^2)\backslash K_1$
is contactomorphic to $(J^{1}(S^{1}), \xi_0)$, cf. \cite[Lemma 7]{DG3},
using a contact flow, we may push $K_2$ into $\Int (N_0)$.

Now let $K$ and $K'$ be two oriented Legendrian torus knots in $(S^1\times S^2,\xi_{\st})$
which have the same oriented knot type and classical invariants $\tw$ and $\rot$.
By the preceding paragraph, we may assume that $K$ and $K'$ are Legendrian torus knots
in $N_0$ with the same invariants $\tw$ and $\rot$.  Use $m_0',l_0'$ to define $(p,q)$-torus
knots in $N_0$. Then a $(p,q)$-torus knot in $N_0$ is also a $(p,q)$-torus knot
in $S^1\times S^2$. Without loss of generality, we may assume that $K$ is a
Legendrian $(p,q)$-torus knot in $N_0$ and $K'$ is a Legendrian $(p',q)$-torus
knot in $N_0$ with $q\ge 2$. By Proposition \ref{topological classification}, $p'\equiv p\mod 2q$ or $p'\equiv -p\mod 2q$.

If $p'\equiv p\mod 2q$, then by interchanging the roles of $K$ and $K'$ if necessary,
we may assume that $p'=p+2kq$, where $k$ is a non-negative integer. There is a contactomorphism $g$
of $(S^1\times S^2,\xi_{\st})$ which sends $K_0$ to $S_+S_-(K_0)$ and is contact isotopic to the
identity (cf. the proof of Theorem \ref{S^1 knot}). We may assume that $g$ sends $N_0$
into $\Int (N_0)$. The class of a parallel copy of $S_{+}S_{-}(K_{0})$ determined by the contact framing is
$l_0'-2m_0'$. Then $g$ sends a $(p_0,q)$-torus knot (corresponding to $p_{0}m_0'+ql_0'$)
in $N_0$ to a $(p_{0}-2q,q)$-torus knot (corresponding to $p_0m_0'+q(l_0'-2m_0')=(p_0-2q)m_0'+ql_0'$)
in $N_0$. Hence $g^k(K')$ is a $(p,q)$-torus knot in $N_0$.
By Proposition \ref{classification in a solid torus} (1), $g^k(K')$ and $K$
are Legendrian isotopic in $ (N_0,\xi_{\st})$. Thus $K$ and $K'$ are Legendrian isotopic in $(S^1\times S^2,\xi_{\st})$.

If $p'\equiv -p\mod 2q$, let $T'$ be a torus in $\Int (N_0)$ parallel to $\partial N_0$
on which $K'$ sits. Let $N_0'$ be the solid torus in $\Int (N_0)$ which has boundary $T'$.
Let $K_0'$ be an oriented Legendrian core of $N_0'$ such that $K_0'$ is isotopic to
$S^1\times \{ (0,0,1)\}$ in $S^1\times S^2$ as oriented knots. Stabilize $K_0'$ if necessary
to make $\rot (K_0')=\rot (K_0)=\rot (K_1)$. By Proposition \ref{classification in a solid torus} (2),
$K_0'$ is Legendrian isotopic to $S_ +^kS_-^k(K_0)$ in $(N_0,\xi_{\st})$, where $k$ is a
non-negative integer. There is a contactomorphism $h$ of $(S^1\times S^2,\xi_{\st})$ which
sends $K_0'$ to $K_1$ and is contact isotopic to the identity (cf. the proof of
Theorem \ref{S^1 knot}). Using a contact flow in $(S^1\times S^2)\backslash K_1$,
we may assume that $h(K')$ is in $\Int (N_0)$. Since $K_0'$ is Legendrian isotopic
to $S_{+}^{k} S_{-}^{k}(K_{0})$ in $(N_0,\xi_{\st})$, the class of a parallel copy
of $K_0'$ determined by the contact framing is $l_0'-2km_0'$. Note that a meridian of $K_1$
corresponds to $-m_0'$. Since $h$ sends $K_0'$ to $K_1$ and is a contactomorphism, $h$
sends $m_0'$ to $-m_0'$ and sends $l_0'-2km_0'$ to $l_0'$, thus sends $l_0'$ to $l_0'-2km_0'$,
and hence sends the $(p',q)$-torus knot $K'$ (corresponding to $p'm_0'+ql_0'$) to a $(-p'-2kq,q)$-torus
knot (corresponding to $p'(-m_0')+q(l_0'-2km_0')=(-p'-2kq)m_0'+ql_0'$) in $N_0$.
By the preceding paragraph, $h(K')$ and $K$ are Legendrian isotopic in $(S^1\times S^2,\xi_{\st})$.
Thus $K$ and $K'$ are Legendrian isotopic in $(S^1\times S^2,\xi_{\st})$.
\end{proof}

\begin{figure}[htb]
\begin{overpic}
{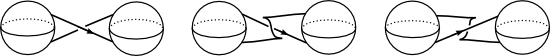}
\put(60, -5){$L_{0}$}
\put(370, -5){$L_{1}$}
\put(220, -5){$L_{-1}$}
\end{overpic}
\caption{Three Legendrian torus knots in $(S^1\times S^2,\xi_{\st})$.}
\label{fig:3}
\end{figure}

\begin{figure}[htb]
\begin{overpic}
{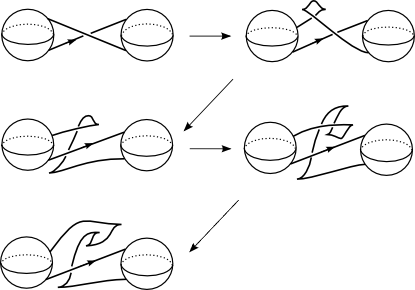}
\end{overpic}
\caption{A Legendrian isotopy.}
\label{fig:4}
\end{figure}

\begin{example}
As described in \cite[Section 2]{Go}, one can represent
$(S^1\times S^2,\xi_{\st})$ by the Kirby diagram with one
$1$-handle in the standard contact structure on $S^3$. We define
the rotation number of an oriented Legendrian knot in $(S^1\times
S^2,\xi_{\st})$ as described in \cite[p. 635]{Go}. In
Figure~\ref{fig:3}, $L_0$ is a Legendrian $(2,1)$-torus knot with
$\rot(L_0)=0$ and $\tw (L_0)=-1$, $L_{-1}$ is a Legendrian
$(2,-1)$-torus knot with $\rot (L_{-1})=-1$ and $\tw (L_{-1})=0$,
and $L_1$ is a Legendrian $(2,-1)$-torus knot with $\rot (L_1)=1$
and $\tw (L_1)=0$. By Proposition \ref{topological
classification}, $L_0,L_{-1},L_1$ are isotopic in $S^1\times S^2$.
Furthermore, by Theorem \ref{torus knot}, $L_0$ is Legendrian
isotopic to $S_+(L_{-1})$ and $S_-(L_1)$. An explicit Legendrian
isotopy between $L_0$ and $S_-(L_1)$ is shown in
Figure~\ref{fig:4}. In the second step of this Legendrian isotopy,
we perform a move of type $6$ from \cite[Theorem 2.2]{Go}.
\end{example}

We give some propositions on the invariants of Legendrian
$(p,q)$-torus knots, $q\geq 2$, in $(S^1\times S^2, \xi_{\st})$.

\begin{proposition}
For a Legendrian $(p,q)$-torus knot, $q\geq2$, in $(S^1\times S^2,
\xi_{\st})$, the maximal $\tw$ invariant is $0$.
\end{proposition}

\begin{proof}
According to the proof of Theorem~\ref{torus knot}, we can push a
Legndrian $(p,q)$-torus knot in $(S^1\times S^2, \xi_{\st})$ into
a Legendrian $(p',q)$-torus knot in $(N_0, \xi_{\st})$. According
to Section~3, if $p'>0$, then it has $\tw\leq -p'<0$, and if
$p'<0$, then it has $\tw\leq 0$. Note that the $\tw$ invariant of
a Legndrian $(p,q)$-torus knot in $(S^1\times S^2, \xi_{\st})$
coincides with that of its push-off in $(N_0, \xi_{\st})$. So the
maximal $\tw$ invariant of a Legendrian $(p,q)$-torus knot in
$(S^1\times S^2, \xi_{\st})$ is nonpositive. On the other hand,
there exists a Legendrian $(p',q)$-torus knot in $(N_0,
\xi_{\st})$, with $p'\equiv p\mod 2q$ and $p'<0$, which has $\tw$
invariant $0$. So the proposition holds.
\end{proof}

\begin{proposition}
For a Legendrian $(p,q)$-torus knot, $q\geq2$, in $(S^1\times S^2,
\xi_{\st})$, with the maximal $\tw$ invariant, it has $\rot\in\{
\pm p+2dq: d\in \Z\}$.
\end{proposition}

\begin{proof}
There exists an integer $d'$ such that $\pm p-2d'q<0$. Choose a
Legendrian $(p-2d'q,q)$-torus knot in $(N_0, \xi_{\st})$ which has
maximal $\tw$ invariant. Then, according to Section~3, its
rotation number belongs to $\{ \pm( p-2d'q+2d''q): 0\leq
d''<\frac{-p+2d'q}{q}, d''\in \Z\}$. Choose a Legendrian
$(-p-2d'q,q)$-torus knot in $(N_0, \xi_0)$ which has maximal $\tw$
invariant. Then its rotation number belongs to $\{ \pm(
-p-2d'q+2d''q): 0\leq d''<\frac{p+2d'q}{q}, d''\in \Z\}$. Both of
these two Legendrian knots are Legendrian $(p,q)$-torus knots in
$(S^1\times S^2, \xi_{\st})$. Note that the rotation number of a
Legendrian $(p,q)$-torus knot in $(S^1\times S^2, \xi_{\st})$
coincides with that of its push-off in $(N_0, \xi_{\st})$. Since
$d'$ can be arbitrarily large, the rotation number of a Legendrian
$(p,q)$-torus knot in $(S^1\times S^2, \xi_{\st})$ with maximal
$\tw$ invariant can be, and can only be, any number of $\{\pm
p+2dq: d\in \Z\}$.
\end{proof}

\begin{proposition}
A Legendrian $(p,q)$-torus knot, $q\geq2$, in $(S^1\times S^2, \xi_{\st})$
with non-maximal $\tw$ can be destabilized in $(S^1\times S^2, \xi_{\st})$.
\end{proposition}

\begin{proof}
We can push a Legendrian $(p,q)$-torus knot in $(S^1\times S^2,
\xi_{\st})$ to be a Legendrian $(p',q)$-torus knot in $(N_0,
\xi_{\st})$ such that $p'\equiv p\mod 2q$ and $p'<0$. Then the
$\tw$ invariant of the Legendrian $(p',q)$-torus knot is
non-maximal. According to Section~3, we can destabilize it in
$(N_0, \xi_{\st})$. So we can destabilize the Legendrian
$(p,q)$-torus knot in $(S^1\times S^2, \xi_{\st})$.
\end{proof}


\begin{thebibliography}{1dfsgg}

\bibitem[BE]{BE} K. L. Baker and J. B. Etnyre, \emph{Rational linking
and contact geometry}, Perspectives in Analysis, Geometry, and
Topology, Progr. Math. 296 (Birkh\"{a}user, Basel, 2012), 19-37.

\bibitem[DG1]{DG1} F. Ding and H. Geiges, \emph{Legendrian knots and links classified by classical
invariants}, Commun. Contemp. Math. 9 (2007), 135-162.

\bibitem[DG2]{DG2} F. Ding and H. Geiges, \emph{Legendrian helix and
cable links}, Commun. Contemp. Math. 12 (2010), 487-500.

\bibitem[DG3]{DG3} F. Ding and H. Geiges, \emph{The diffeotopy group of
$S^1\times S^2$ via contact topology}, Compositio Math. 146 (2010),
1096-1112.

\bibitem[EF1]{EF1} Y. Eliashberg and M. Fraser, \emph{Classification of
topologically trivial Legendrian knots}, Geometry, Topology, and
Dynamics (Montr\'{e}al, 1995), CRM Proc. Lecture Notes Vol. 15
(Amer. Math. Soc., Providence, 1998), 17-51.

\bibitem[EF2]{EF2} Y. Eliashberg and M. Fraser, \emph{Topologically
trivial Legendrian knots}, J. Symplectic Geom. 7 (2009), 77-127.

\bibitem[Et]{Et} J. B. Etnyre, \emph{Legendrian and transversal knots},
Handbook of Knot Theory (Elsevier, Amsterdam, 2005), 105-185.

\bibitem[EH]{EH} J. B. Etnyre and K. Honda, \emph{Knots and contact
geometry I: Torus knots and the figure eight knot}, J. Symplectic
Geom. 1 (2001), 63-120.

\bibitem[ELT]{ELT} J. B. Etnyre, D. J. LaFountain and B. Tosun,
\emph{Legendrian and transverse cables of positive torus knots}, Geom.
Topol. 16 (2012), no. 3, 1639-1689.

\bibitem[ENV]{ENV} J. B. Etnyre, L. L. Ng and V. V\'{e}rtesi,
\emph{Legendrian and transverse twist knots}, J. Eur. Math. Soc. (JEMS), 15 (2013), no. 3, 969-995.

\bibitem[Ge]{Ge} H. Geiges, \emph{An introduction to contact topology},
Cambridge Studies in Advanced Mathematics, vol. 109 (Cambridge
University Press, Cambridge, 2008).

\bibitem[GO]{GO} H. Geiges and S. Onaran, \emph{Legendrian rational
unknots in lens spaces}, arXiv:1302.3792.

\bibitem[Gh]{Gh} P. Ghiggini, \emph{Linear Legendrian curves in $T^3$},
Math. Proc. Camb. Philos. Soc. 140 (2006), 451-473.

\bibitem[Go]{Go} R. E. Gompf, \emph{Handlebody construction of Stein
surfaces}, Ann. of Math. (2) 148 (1998), 619-693.

\bibitem[Ha]{Ha} A. Hatcher, \emph{Notes on Basic 3-Manifold
Topology}, available at
http://www.math.cornell.edu$/^{\sim}$hatcher.

\bibitem[HR]{HR} C. Hodgson and J. H. Rubinstein, \emph{Involutions and
isotopies of lens spaces}, Knot theory and manifolds (Vancouver,
1983), Lecture Notes in Mathematics, vol. 1144, ed. D. Rolfsen
(Springer, Berlin, 1985), 60-96.

\bibitem[On]{On} S. C. Onaran, \emph{Legendrian knots in lens spaces},
arXiv:1012.3047.


\end{thebibliography}
\end{document}